\documentclass[12pt, reqno]{amsart}
\usepackage{amsmath, amsthm, amscd, amsfonts, amssymb, graphicx, color}
\usepackage[bookmarksnumbered, colorlinks, plainpages]{hyperref}

\makeatletter \oddsidemargin.9375in \evensidemargin \oddsidemargin
\def\Correspondingauthor{$^{*}$\protect\footnotetext{$^{*}$ C\lowercase{orresponding author.}}}
\def\authorsaddresses#1{\dedicatory{#1}}
\newtheorem{theorem}{Theorem}[section]
\newtheorem{lemma}[theorem]{Lemma}
\newtheorem{proposition}[theorem]{Proposition}

\theoremstyle{definition}
\newtheorem{definition}[theorem]{Definition}

\theoremstyle{remark}

\numberwithin{equation}{section}

\begin{document}
\setcounter{page}{1}


\title[Invertibility of Generalized Bessel multipliers]{Invertibility of Generalized Bessel multipliers in Hilbert $C^{*}$-modules}

\author[Abbaspour Tabadkan, Hossein-nezhad]{Gholamreza Abbaspour Tabadkan$^1$\Correspondingauthor, Hessam Hossein-nezhad$^1$}

\authorsaddresses{$^1$ Department of pure mathematics, school of mathematics and computer science,
Damghan university, Damghan, Iran.\\
abbaspour@du.ac.ir; h.hosseinnezhad@std.du.ac.ir}

\subjclass[2010]{Primary 42C15 Secondary  46C05, 47A05}

\keywords{Bessel multiplier; Modular Riesz basis; Standard frame.}

\begin{abstract}
In this note, a general version of Bessel multipliers in Hilbert $C^*$-modules is presented and then, many
results obtained for multipliers are extended. Also the conditions for
invertibility of generalized multipliers are investigated in details.
The invertibility of multipliers is very important because it helps us
to obtain more reconstruction formula.

\end{abstract}

\maketitle

\section{Introduction}
Frames in Hilbert space were originally introduced by Duffin and Schaeffer \cite{R.J. Duffin and A.C. Schaeffer-1952} to deal with some problems in nonharmonic Fourier analysis.
Many generalizations of frames were introduced, e.g. pseudo-frames,
oblique frames, G-frames, and fusion frames(frames of subspaces).

Frank and Larson \cite{Frank-Larson-2002} extended the frame
theory for the elements of $C^{*}$-algebras and (finitely or countably generated) Hilbert $C^{*}$-modules.
Extending the results to this
more general framework is not a routine generalization, as there are essential differences
between Hilbert $C^{*}$-modules and Hilbert spaces.
For example, we know that the Riesz representation theorem for continuous linear functionals on Hilbert spaces
dose not extend to Hilbert $C^{*}$-modules and there exist closed subspaces
in Hilbert $C^{*}$-modules that have no orthogonal complement. Moreover,
we know that every bounded operator on a Hilbert space has an adjoint, while there are bounded operators on Hilbert $C^{*}$-modules which
do not have any.

Bessel multipliers in Hilbert spaces were introduced by Balazs in \cite{Balazs-2007}. Bessel
multipliers are operators that are defined by a fixed multiplication pattern which
is inserted between the analysis and synthesis operators. This class of operators is not only of interest for applications in modern life,
for example in
acoustics, psychoacoustics and denoising,
but also it is important in different branches of functional analysis.
Recently, M. Mirzaee Azandaryani and A. Khosravi generalized multipliers to Hilbert $C^{*}$-modules \cite{Khosravi-Mirzaee Azandaryani}.

The standard matrix description of operators on Hilbert spaces, using an orthonormal basis,
was presented in \cite{J.B. Conway}. This idea was developed for Bessel sequences, frames and Riesz sequences by Balazs \cite{Balazs-2008}.
In the last paper, the author also studied the dual function, which assigns an operator to a matrix.
Using this approach, a generalization of Bessel multipliers is obtained, as introduced in \cite{Balazs-2009}.
In \cite{abbaspour}, the authors investigated some properties of generalized multipliers in details.
In this paper, we are going to extend this concept to Hilbert modules.

The paper is organized as follows.

In section $2$, some notations and preliminary results of Hilbert modules, their frames and Bessel multipliers are given.
Section $3$ is devoted to the generalization of Bessel multipliers in Hilbert $C^{*}$-modules and then some conditions for invertibility of
such operators are obtained.
In the last section, we consider generalized modular Riesz multipliers and extend some known results.
Moreover, we add some new consequences of them.

\section{Notation and preliminaries}
In this section, we recall some definitions and basic properties of Hilbert
$C^{*}$-modules and their frames.
Throughout this paper, $A$ is a unital $C^{*}$-algebra and $E$, $F$ are finitely or countably generated Hilbert $A$-modules.

\label{sec:intro}
A (left) Hilbert $C^{*}$-module over the $C^{*}$-algebra A
is a left $A$-module $E$ equipped with an $A$-valued inner product
$\langle\cdot, \cdot\rangle : E\times E\rightarrow A$ satisfying the following
conditions:
\begin{enumerate}
\item
$\langle x,x\rangle\geq 0$ for every $x\in E$ and $\langle x,x\rangle =0$ iff $x=0$,
\item
$\langle x,y\rangle =\langle y,x\rangle^*$ for every $x,y\in E$,
\item
$\langle\cdot,\cdot\rangle$ is $A$-linear in the first argument,
\item
$E$ is complete with respect to the norm $\| x\|^2=\|\langle x,x\rangle\|_A$.
\end{enumerate}
Given Hilbert $C^{*}$-modules $E$ and $F$, we denote by $L(E, F)$ the set of all adjointable operators from $E$ to $F$
(i.e. of all maps $T: E\rightarrow F$ such that there exists $T^*: F\rightarrow E $ with the property
$\langle Tx, y\rangle = \langle x, T^*y\rangle$ for all $x\in E, y\in F$).
It is well-known that each adjointable operator is necessarily bounded and $A$-linear in the sense $T(ax) = aT(x)$, for all $a\in A, x\in E$.

For each elements $x\in E, y\in F$, we define the operator
$\Theta_{x,y}:E\rightarrow F$
by $\Theta_{x,y}(z)=\langle z, x\rangle y$, for each $z\in E$.
It is easy to check that
$\Theta_{x,y}\in L(E, F)$ and $(\Theta_{x,y})^* = \Theta_{y,x}$. Operators of this form are called elementary operators.
Each finite linear combination of elementary operators is
said to be a finite rank operator. The closed linear span of the set $\{\Theta_{x,y}:~x\in F,~y\in E\}$ in $L(E, F)$ is
denoted by $K(E,F)$ and its elements will be called compact operators.
Specially, if $E=F$, we write $L(E)$ and $K(E)$, respectively. It is well-known that $L(E)$ is a $C^{*}$-algebra
and $K(E)$ is the closed two-sided ideal in $L(E)$.
Recall that the center of a Banach algebra $A$, denoted $Z(A)$, is defined as
$Z(A)=\{a\in A; ab = ba, \forall b\in A\}$.
It is clear that if $a\in Z(A)$, then $a^*\in Z(A)$, also if $a$ is a positive element of $Z(A)$, then $a^{\frac{1}{2}}\in Z(A)$.

Let $A$ be a $C^{*}$-algebra. Consider
$$\ell^2(A) := \{\{a_n\}_n\subseteq A:~\sum_{n}a_n a^{*}_{n}\text{converges in norm in}~A\}.$$
It is easy too see that $\ell^2(A)$ with pointwise operations and the inner product
$$\langle\{a_n\}, \{b_n\}\rangle=\sum_{n}a_n b^{*}_{n},$$
becomes a Hilbert $C^{*}$-module which is called the standard Hilbert $C^{*}$-module over $A$.
A Hilbert $A$-module $E$ is called finitely generated (resp. countably generated) if there exist a finite subset $\{x_1, ..., x_n\}$ (resp. countable set $\{x_n\}_{n}$)
of $E$ such that $E$ equals the closed $A$-linear hull of this set.
For more details about Hilbert $C^{*}$-modules, we refer the interested
reader to the books \cite{Lance-1995, Manuilov-Troitsky}.

Now, we recall the concept of frame in Hilbert $C^{*}$-modules which is defined in \cite{Frank-Larson-2002}.
Let $E$ be a countably generated Hilbert module over a unital $C^{*}$-algebra $A$. A sequence $\{x_n\}\subset E$ is said
to be a \emph{frame} if there exist two constant $C, D>0$ such that
\begin{equation}\label{frame intro}
C\langle x, x\rangle\leq\sum_{n}\langle x, x_n\rangle\langle x_n, x\rangle\leq D\langle x, x\rangle
\end{equation}
for every $x\in E$.
The optimal constants (i.e. maximal for $C$ and minimal for $D$) are called frame bounds. If the sum in \eqref{frame intro} converges in norm, the
frame is called \emph{standard frame}.
The sequence $\{x_n\}$ is called a \emph{Bessel sequence} with bound $D$ if the upper inequality in \eqref{frame intro} holds for every $x\in E$.

Suppose that $\{x_n\}$ is a standard frame of a Hilbert
$A$-module $E$ with bounds $C$ and $D$. The operator $T:E\rightarrow\ell^2(A)$ defined by
\begin{equation*}
Tx = \{\langle x, x_n\rangle\}_n,
\end{equation*}
is called the analysis operator. The adjoint operator $T^{*}:\ell^2(A)\rightarrow E$ is given by
\begin{equation*}
T^{*}(\{a_n\}) = \sum_{n}a_n~.~x_n.
\end{equation*}
$T^{*}$ is called the synthesis operator. By composing $T$ and $T^{*}$, we obtain the frame operator $S:E\rightarrow E$ as:
\begin{equation*}
Sx = T^{*}Tx = \sum_{n}\langle x, x_n\rangle x_n.
\end{equation*}
The operator $S$ is well-defined, positive, invertible and adjointable; moreover it satisfies $C\leq S\leq D$ and
$D^{-1}\leq S^{-1}\leq C^{-1}$. Also for each $x\in E$, we have \emph{the reconstruction formula} as follows:
\begin{equation}\label{recons formula}
x = \sum_{n}\langle x, S^{-1}x_n\rangle x_n = \sum_{n}\langle x, x_n\rangle S^{-1}x_n.
\end{equation}
The sequence $\{\tilde{x}_n\} = \{S^{-1}x_n\}$, which is a standard frame with bounds $D^{-1}$ and $C^{-1}$, is called the \emph{canonical dual frame} of $\{x_n\}$.
Sometimes the reconstruction formula of standard frames is valid
with other (standard) frames $\{y_n\}$ instead of $\{S^{-1}x_n\}$. They are said to be \emph{alternative dual frames} of $\{x_n\}$.

Now let us take a brief review of the definition of Bessel multipliers in Hilbert $C^{*}$-modules.

Let $E$ and $F$ be two Hilbert modules over a unital $C^{*}$-algebra $A$, and let $\{x_n\}\subseteq E$ and $\{y_n\}\subseteq F$ be
standard Bessel sequences. Moreover let $m=\{m_n\}\in \ell^{\infty}(A)$ be such that $m_n\in Z(A)$, for each $n$, and $\mathcal{M}_{m}$
defined on $\ell^2(A)$ as $\mathcal{M}_{m}(\{a_n\}) = \{m_n a_n\}$.

The operator $\mathbf{M}_{m,\{y_n\},\{x_n\}}:E\rightarrow F$ which is defined by
\begin{equation}
\mathbf{M}_{m,\{y_n\},\{x_n\}} = T^{*}_{\{y_n\}}\mathcal{M}_{m}T_{\{x_n\}},
\end{equation}
is called the \emph{Bessel multiplier} for the Bessel sequences $\{x_n\}$ and $\{y_n\}$. It is easy to see that
$\mathbf{M}_{m,\{y_n\},\{x_n\}}(x) = \sum_{n}m_n\langle x, x_n\rangle y_n$.
For more details about the Bessel multipliers in Hilbert $C^{*}$-modules, one can see \cite{Khosravi-Mirzaee Azandaryani}.
\section{Generalized Bessel multipliers in Hilbert $C^{*}$-modules}
The matrix representation of operators in Hilbert spaces using an orthonormal basis \cite{J.B. Conway}, Gabor
frames \cite{Grochenig-2006} and linear independent Gabor systems \cite{Strohmer-2006} led Balazs to develop 
this idea in full generality for Bessel sequences, frames and
Riesz sequences \cite{Balazs-2008}.
In the same paper, the author also established
the function which assigns an operator in $\mathcal{B}(\mathcal{H}_1, \mathcal{H}_2)$ to an infinite matrix in $\mathcal{B}(\ell^2)$.
The last concept is a generalization of Bessel multiplier as introduced in \cite{Balazs-2008}. The following essential definition is recalled from \cite{Balazs-2008, Balazs-2009}.
\begin{definition}
Let $\mathcal{H}_1$ and $\mathcal{H}_2$ be Hilbert spaces and $X=\{x_n\}\subset\mathcal{H}_1$ and $Y=\{y_n\}\subset\mathcal{H}_2$ be
Bessel sequences. Moreover let $M$ be an infinite matrix defining a bounded operator from $\ell^2$ to $\ell^2$, $(Mc)_{i} = \sum_{k}M_{i,k}c_k$.
Then the operator $\mathcal{O}^{(X,Y)}(M):\mathcal{H}_1\rightarrow\mathcal{H}_2$ defined by
\begin{equation*}
(\mathcal{O}^{(X,Y)}(M))h = T^{*}_{Y}MT_{X} (h) = \sum_k \sum_j M_{k,j} \langle h, x_j\rangle y_k,~~~(h\in\mathcal{H}_1),
\end{equation*}
is called \emph{the generalized Bessel multiplier} for the Bessel sequences $X$ and $Y$.
\end{definition}
In the sequel, first we introduce the concept of Generalized Bessel multipliers for countably generated Hilbert $C^{*}$-modules
and then, we will discuss some properties of such operators.
\begin{definition}
Let $E$ and $F$ be two Hilbert $C^{*}$-modules over a unital $C^{*}$-algebra $A$ and $X=\{x_n\}\subset E$ and $Y=\{y_n\}\subset F$ be
standard Bessel sequences. Also let $U\in L(\ell^2(A))$ be an arbitrary non-zero operator.
The operator $\mathbf{M}_{U,Y,X}:E\rightarrow F$ which is defined as
\begin{equation}\label{generalized Bessel multiplier}
\mathbf{M}_{U,Y,X}(x) = T^{*}_{Y}UT_{X}(x)~~~~(x\in E),
\end{equation}
is called the \emph{Generalized Bessel multiplier} associated to $X$ and $Y$ with symbol $U$.
Some of the main properties of the generalized Bessel multipliers are summarized in the next proposition.
\end{definition}
\begin{proposition}\label{properties of multi}
For the generalized Bessel multipliers $\mathbf{M}_{U,Y,X}$, the following assertions hold:
\begin{enumerate}
  \item $\mathbf{M}_{U,Y,X}\in L(E, F)$ and $\mathbf{M}^{*}_{U,Y,X} = \mathbf{M}_{U^{*},X,Y}$.
  \item If $U$ is a compact operator on $\ell^2(A)$, then $\mathbf{M}_{U,Y,X}\in K(E, F)$.
  \item If $U$ is a positive operator on $\ell^2(A)$, then $\mathbf{M}_{U,X,X}\in L(E)$ is a positive operator.
\end{enumerate}
\end{proposition}
\begin{proof}
(1) It is clear that $\mathbf{M}_{U,Y,X}\in L(E, F)$. Also
$$\mathbf{M}^{*}_{U,Y,X} = (T^{*}_{Y}UT_{X})^{*} = T^{*}_{X}U^{*}T_{Y} = \mathbf{M}_{U^{*},X,Y}.$$\\
(2) At the first, let us prove that $\mathbf{M}_{U,Y,X}$ is a finite rank operator if $U$ is one. If $U$ is a finite rank operator, then
$U = \sum_{j=1}^{n} \Theta_{a_j, b_j}$, for some $a_j, b_j\in \ell^2(A),(j=1, ..., n)$. Hence,
$$\mathbf{M}_{U,Y,X} = T^{*}_{Y}UT_{X} = T^{*}_{Y}~(\sum_{j=1}^{n} \Theta_{a_j, b_j})~T_{X} = \sum_{j=1}^{n}\Theta_{T^{*}_{X}a_j, T^{*}_{Y}b_j}.$$
Therefore, $\mathbf{M}_{U,Y,X}$ is a finite rank operator from $E$ to $F$. Now let $U$ be a compact operator on $\ell^2(A)$. Thus for each
$\epsilon>0$, there exists a sequence of finite rank operators on $\ell^2(A)$, say $\{U_{\alpha}\}$, such that $\|U_{\alpha} - U\|<\epsilon$.
So
$$\|\mathbf{M}_{U_{\alpha},Y,X} - \mathbf{M}_{U,Y,X}\|\leq \|T^{*}_{Y}\|~\|U_{\alpha} - U\|~\|T_{X}\|\leq \sqrt{DD'}\epsilon.$$
As seen above, $\mathbf{M}_{U_{\alpha},Y,X}$ are finite rank. From this facts, we conclude that  $\mathbf{M}_{U,Y,X}$ is a compact operator.\\
(3) Since $U$ is positive, by \cite[Lemma 4.1]{Lance-1995}, $\langle a, Ua\rangle\geq 0$ for all $a = \{a_n\}\in \ell^2(A)$. So
$$\langle x, \mathbf{M}_{U,X,X}x\rangle = \langle x, T^{*}_{X}UT_{X}x \rangle = \langle T_{X}x, UT_{X}x\rangle\geq 0.$$
Again by \cite[Lemma 4.1]{Lance-1995}, it follows that $\mathbf{M}_{U,X,X}$ is positive.
\end{proof}
The following proposition shows that if
one of the sequences is standard Bessel sequence, invertibility of multiplier implies that the other one satisfies the lower frame condition.
\begin{proposition}
Let $X = \{x_n\}\subset E$ be a standard Bessel sequence with upper bound $D$ and $Y = \{y_n\}\subset F$ be an arbitrary sequence.
If $\mathbf{M}_{U,Y,X}$ is an invertible operator,
then $Y = \{y_n\}$ satisfies the lower frame condition.
\end{proposition}
\begin{proof}
For each $x\in E, y\in F$:
\begin{equation*}
\begin{split}
\|\langle \mathbf{M}_{U,Y,X}(x), y\rangle\|_{A}
&= \|\langle T^{*}_{Y}UT_{X}(x), y\rangle\|_{A}\\
&\leq \|UT_{X}(x)\|_{\ell^2(A)}~\|T_{Y}(y)\|_{\ell^2(A)}\\
&\leq \sqrt{D}\|U\|~\|x\|\|\{\langle y, y_n\rangle\}_{n}\|_{\ell^2(A)}\\
&= \sqrt{D}\|U\|~\|x\|~\|\sum_{n}\langle y, y_n\rangle\langle y_n, y\rangle\|^{1/2}.\\
\end{split}
\end{equation*}
Put $x = \mathbf{M}^{-1}_{U,Y,X}(y)$. Then
$$\dfrac{1}{D\|U\|^2\|\mathbf{M}^{-1}_{U,Y,X}\|^2}\|y\|^{2}\leq\|\sum_{n}\langle y, y_n\rangle\langle y_n, y\rangle\|.$$
Therefore by \cite[Proposition 3.8]{Jing-2006} , we conclude that $Y = \{y_n\}$ satisfies the lower frame condition and the proof is complete.
\end{proof}
Similar to the case of operators on Hilbert spaces, we also have the following perturbation
result for Hilbert modules. In the sequel, we will use this result on several occasions.
\begin{lemma}\label{perturbation lemma}
Let $E$ be a Hilbert $A$-module and $U:E\rightarrow E$ be an invertible operator on $E$. Also let $W\in L(E)$ be such that for each $x\in E$,
$\|Ux-Wx\|\leq \lambda\|x\|$ where $\lambda\in[0, \|U^{-1}\|^{-1})$. Then $W$ is invertible and
$$\dfrac{1}{\lambda+\|U\|}\|x\|\leq\|W^{-1}x\|\leq\dfrac{1}{\|U^{-1}\|^{-1}-\lambda}\|x\|.$$
\end{lemma}
\begin{proof}
It follows directly from the proofs of \cite[Theorem 3.2.3]{Aupetit-1991} and \cite[Proposition 2.2]{Stoeva-2012}.
\end{proof}
The next proposition investigates some sufficient conditions for invertibility of generalized frame multipliers.
\begin{proposition}
Let $E$ be a Hilbert $A$-module and $\{x_n\}$ be a standard frame
for $E$ with bounds $C$ and $D$. Suppose that $\{y_n\}$ is a sequence of $E$ and there exists a positive constant $\lambda<\dfrac{1}{D}\Big(\dfrac{CD^2-C^2D}{C^2+D^2}\Big)^2$
such that
\begin{equation}
\|\sum_{n} \langle x, x_n - y_n\rangle\langle x_n - y_n, x\rangle\|\leq \lambda\|x\|^{2}.
\end{equation}
Moreover, suppose that $U$ is a non-zero adjointable operator on $\ell^2(A)$ with $\|U-I\|<\dfrac{C^2}{D^2}$.
Then $\{y_n\}$ is a standard frame and $\mathbf{M}_{U,X,Y}$ is invertible.
\end{proposition}
\begin{proof}
The first part follows from \cite[Theorem 3.2]{Han-Jing-Mohapatra}. Now, let us deal with the second claim.
suppose  $S_{X}$ is the frame operator associated to $X=\{x_n\}$. For each $x\in E$:
 \begin{equation*}
\begin{split}
\|\mathbf{M}_{U,X,X}(x) - S_{X}(x)\|
&= \|\mathbf{M}_{U,X, X}(x) - \mathbf{M}_{I,X,X}(x)\|\\
&= \|\mathbf{M}_{U-I,X,X}(x)\|\\
&= \|T^{*}_{X}~(U-I)~T_{X}(x)\|\\
&\leq \|T^{*}_{X}\|~\|U-I\|~\|T_{X}(x)\|\\
&< (C^2/D)\|x\|.\\
\end{split}
\end{equation*}
So by Lemma \ref{perturbation lemma}, $\mathbf{M}_{U,X,X}$ is an invertible operator with
$$\dfrac{1}{\|S_{X}\|+C^2/D}\leq\|\mathbf{M}^{-1}_{U,X,X}\|\leq \dfrac{1}{\|S^{-1}_{X}\|^{-1}-C^2/D}~.$$
Now for every $x\in E$,
\begin{equation*}
\begin{split}
\|\mathbf{M}_{U,X,Y}(x) - \mathbf{M}_{U,X,X}(x)\|
&= \|\mathbf{M}_{U,X,Y-X}(x)\|\\
&= \|T^{*}_{X}~U~T_{Y-X}(x)\|\\
&\leq \|T^{*}_{X}\|~\|U\|~\|T_{Y-X}(x)\|\\
&\leq \|U\|\sqrt{D}\sqrt{\lambda}\|x\|.\\
\end{split}
\end{equation*}
If we show that $\|U\|\sqrt{D}\sqrt{\lambda}<\dfrac{1}{\|\mathbf{M}^{-1}_{U,X,X}\|}$, then the proof will be completed. But
$$\|U\|\sqrt{D}\sqrt{\lambda}\leq\dfrac{C^2+D^2}{D^2}\sqrt{D}\sqrt{\lambda}<C-\dfrac{C^2}{D}\leq
\|S^{-1}_{X}\|^{-1}-\dfrac{C^2}{D}\leq\dfrac{1}{\|\mathbf{M}^{-1}_{U,X,X}\|},$$
and so by Lemma \ref{perturbation lemma} the result holds.
\end{proof}
The following two propositions contain sufficient conditions for the invertibility of frame multipliers.
\begin{proposition}
Let $Y=\{y_n\}$ be a standard frame for Hilbert $A$-module $E$ with bounds $C$ and $D$, $~W:E\rightarrow E$ be an adjointable and bijective operator
and $x_n = W(y_n)$ for each $n$. Moreover let $U$ be a bounded operator on $\ell^2(A)$ such that $\|U-I\|<C/D$. Then the following statements hold:
\begin{enumerate}
  \item $X = \{x_n\}$ is a standard frame for $E$.
  \item $\mathbf{M}_{U,Y,X}$$(\text{resp.}~\mathbf{M}_{U,X,Y})$ is invertible and $\mathbf{M}^{-1}_{U,Y,X} = (W^{-1})^*~\mathbf{M}^{-1}_{U,Y,Y}$
$(\text{resp.}~\mathbf{M}_{U,X,Y}^{-1} = \mathbf{M}^{-1}_{U,Y,Y}(W^{-1}))$.
\end{enumerate}
\begin{proof}
(1) Follows from \cite[Theorem 2.5]{ARAMBASIC-2007}.\\
(2) First note that $\mathbf{M}_{U,Y,X} = \mathbf{M}_{U,Y,Y}~W^*$.Indeed
\begin{equation*}
\begin{split}
\mathbf{M}_{U,Y,Y}~W^*(f)
&= \sum_k \sum_j \langle W^{*}(f), yj\rangle y_k\\
&= \sum_k \sum_j \langle f, W(y_j)\rangle y_k \\
&= \mathbf{M}_{U,Y,X}(f).\\
\end{split}
\end{equation*}
So it is enough to prove that $\mathbf{M}_{U,Y,Y}$ is invertible. For every $x\in E$,
$$\|\mathbf{M}_{U,Y,Y}(x) - S_{Y}(x)\| =
 \|\mathbf{M}_{U-I,Y,Y}(x)\|\leq D\|U-I\|\|x\|<C\|x\|.$$
Since $C\leq\dfrac{1}{\|S^{-1}_{Y}\|}$, it follows from Lemma \ref{perturbation lemma} that $\mathbf{M}_{U,Y,Y}$ is invertible.
The invertibility of $\mathbf{M}_{U,X,Y}$ is obtained with the same argument.
\end{proof}
\end{proposition}
\begin{proposition}
Let $X=\{x_n\}$ be a standard frame for Hilbert $A$-module $E$ with upper bound $D$ and $X^d=\{x^{d}_{n}\}$ be a dual frame of $X$.
Also let $U$ be a bounded operator on $\ell^2(A)$ such that $\|U-I\|<1/2D$. Then the multiplier $\mathbf{M}_{U,X,X^d}$$(\text{resp.}~\mathbf{M}_{U,X^d,X})$
is invertible.
\end{proposition}
\begin{proof}
For every $x\in E$,
$$\|\mathbf{M}_{U,X,X^d}(x) - x\| = \|\mathbf{M}_{U-I,X,X^d}\|\leq D\|U-I\|\|x\|<\dfrac{1}{2}\|x\|.$$
So by Lemma \ref{perturbation lemma}, $\mathbf{M}_{U,X,X^d}$ is invertible.
\end{proof}
\begin{proposition}
Let $Y=\{y_n\}$ be a standard frame for Hilbert $A$-module $E$ with bounds $C$ and $D$ and $\tilde{Y} = \{\tilde{y}_n\}$ be its canonical dual frame.
\begin{enumerate}
  \item If $X = \{x_n\}$ be a standard Bessel sequence such that
  \begin{equation}\label{essential inequality}
  \sum_{n}\|x_n - \tilde{y}_{n}\|^{2}<1/4D,
  \end{equation}
  then $\mathbf{M}_{I,Y,X}$ is invertible.
  \item Let $X = \{x_n\}$ be a standard Bessel sequence and \eqref{essential inequality} holds.
  Also let $U$ be a bounded operator on $\ell^2(A)$ with $\|U\|<1$ and $\|U-I\|<\sqrt{C/4D}$.
  Then $\mathbf{M}_{U,Y,X}$ is invertible.
\end{enumerate}
\end{proposition}
\begin{proof}
(1) For every $x\in E$,
\begin{equation*}
\begin{split}
\|\mathbf{M}_{I,Y,X}(x) - x\|
& = \|T^{*}_{Y}~T_{X}(x) - T^{*}_{Y}~T_{\tilde{Y}}(x)\|\\
& = \|T^{*}_{Y}~T_{X - \tilde{Y}}(x)\|\\
& \leq \sqrt D \|\{\langle x, x_n - \tilde{y}_{n}\rangle\}_n\|_{\ell^2(A)}\\
& = \sqrt D \|\sum_n \langle x, x_n - \tilde{y}_{n}\rangle \langle x_n - \tilde{y}_{n}, x \rangle\|^{1/2}\\
& \leq \sqrt D\Big(\sum_{n}\|x\|^{2}\|x_n - \tilde{y}_{n}\|^{2}\Big)^{1/2}\\
& = \sqrt D\|x\| \Big(\sum_{n} \|x_n - \tilde{y}_{n}\|^{2}\Big)^{1/2}\\
& < \sqrt D (1/2\sqrt D)\|x\| = 1/2\|x\|,\\
\end{split}
\end{equation*}
and so $\mathbf{M}_{I,Y,X}$ is invertible.\\
(2) For every $x\in E$ we have:
\begin{small}
\begin{equation*}
\begin{split}
\|\mathbf{M}_{U,Y,X}(x) - x\|
& \leq \|\mathbf{M}_{U,Y,X}(x) - \mathbf{M}_{U,Y,\tilde{Y}}(x)\|+\|\mathbf{M}_{U,Y,\tilde{Y}}(x) - \mathbf{M}_{I,Y,\tilde{Y}}(x)\|\\
& = \|T^{*}_{Y}UT_{X-\tilde{Y}}(x)\| + \|T^{*}_{Y}(U-I)T_{\tilde{Y}}(x)\|\\
& \leq \sqrt D\|U\|\|T_{X-\tilde{Y}}(x)\| + (\sqrt{D/C}) \|U-I\|\|x\|\\
& \leq \sqrt D\|U\|\|x\|\Big(\sum_n \|x_n - \tilde{y}_{n}\|^2\Big)^{(1/2)} + (\sqrt{D/C}) \|U-I\|\|x\|\\
& < \|x\|.\\
\end{split}
\end{equation*}
\end{small}
Hence we conclude that $\mathbf{M}_{U,Y,X}$ is invertible.
\end{proof}
\section{Generalized modular Riesz multipliers}
The rest of this article is devoted to studying some properties of Riesz multipliers.
For this aim, we borrow the following definition from \cite{A.Khosravi-B.Khosravi}.
\begin{definition}
Let $A$ be a unital $C^{*}$- algebra and $E$ be a finitely or countably generated Hilbert $A$-module. A sequence $\{x_n\}$ is a \emph{modular Riesz basis}
for $E$ if there exists an adjointable and invertible operator $U:\ell^2(A)\rightarrow E$ such that $U(e_n) = x_n$ for each $n$, where $\{e_n\}$
is the orthonormal basis of $\ell^2(A)$.
\end{definition}
The next statement is a generalization of the second part of \cite[Theorem 4.3]{Khosravi-Mirzaee Azandaryani}.
\begin{proposition}
Let $X=\{x_n\}$ and $Y=\{y_n\}$ be modular Riesz bases of Hilbert $A$-modules $E$ and $F$, respectively.
Then the mapping $U\mapsto \mathbf{M}_{U,Y,X}$ is injective from $L(\ell^2(A))$ to $L(E, F)$.
\end{proposition}
\begin{proof}
Suppose that $\mathbf{M}_{U_1,Y,X} = \mathbf{M}_{U_2,Y,X}$. So for each $x\in E$, $\mathbf{M}_{U_1,Y,X}(x)\\=\mathbf{M}_{U_2,Y,X}(x)$.
Thus by definition, we have
$$\sum_{n} \Big(U_1(\{\langle x, x_n\rangle\})\Big)y_n = \sum_{n} \Big(U_2(\{\langle x, x_n\rangle\})\Big)y_n.$$
Since $Y$ is a modular Riesz basis, by \cite[Theorem 3.1]{A.Khosravi-B.Khosravi}, it follows that
$$U_1(\{\langle x, x_n\rangle\}) = U_2(\{\langle x, x_n\rangle\}).$$
Now, since $\{x_n\}$ is a modular Riesz basis, by \cite[Proposition 3.1]{A.Khosravi-B.Khosravi} and \cite[Theorem 4.9]{Jing-2006}, the associated analysis operator is surjective and hence we conclude $U_1 = U_2$.
\end{proof}
In \cite[Lemma 4.1]{Khosravi-Mirzaee Azandaryani}, it is shown that the modular Riesz basis $\{x_n\}$ and its canonical dual $\{\tilde{x}_n\} = \{S^{-1}x_n\}$ form a pair of biorthogonal sequences.
Due to this fact, we check some properties of modular Riesz multipliers.
\begin{proposition}
Let $X = \{x_n\}$ and $Y = \{y_n\}$ be two modular Riesz bases with bounds $C, D$ and $C', D'$, respectively. Then
$$K\sqrt {CC'}\leq \|\mathbf{M}_{U,Y,X}\|\leq \sqrt {DD'}\|U\|,$$
where $K := \sup \{\|U(e_n)\|; \{e_n\} \text{is the ONB for}~ \ell^2(A)\}$.
\end{proposition}
\begin{proof}
The upper inequality follows from Proposition \ref{properties of multi}. Now, for the lower inequality, by chossing the arbitrary
index $n_0$, we have
$$\mathbf{M}_{U,Y,X}(\tilde{x}_{n_0}) = T^{*}_{Y}U(e_{n_0}).$$
So
$$\|\mathbf{M}_{U,Y,X}\|\geq\dfrac{\|\mathbf{M}_{U,Y,X}(\tilde{x}_{n_0})\|}{\|\tilde{x}_{n_0}\|}
 = \dfrac{\|T^{*}_{Y}U(e_{n_0})\|}{\|\tilde{x}_{n_0}\|}\geq K\sqrt{CC'},$$
 and since $n_0$ is chosen arbitrary, the proof is complete.
\end{proof}
The next two propositions give some necessary and sufficient conditions for invertibility of generalized multipliers
associated to modular Riesz bases.
\begin{proposition}
Let $U$ be an bounded linear operator on $\ell^2(A)$ and $X=\{x_n\}$ and $Y=\{y_n\}$ be two modular Riesz bases for Hilbert $A$-module $E$. Then $U$ is invertible if and only if the generalized Riesz
multiplier $\mathbf{M}_{U,Y,X}$ is invertible.
\end{proposition}
\begin{proof}
Let $\tilde{X}$ and $\tilde{Y}$ be the dual modular Riesz bases of $X$ and $Y$, respectively. If $U$ be invertible, then
\begin{equation*}
(\mathbf{M}_{U,Y,X})(\mathbf{M}_{U^{-1},\tilde{X},\tilde{Y}})
 = (T^{*}_{Y}UT_{X})(T^{*}_{\tilde{X}}U^{-1}T_{\tilde{Y}})
 = Id,
\end{equation*}
and similarly $(\mathbf{M}_{U^{-1},\tilde{X},\tilde{Y}})(\mathbf{M}_{U,Y,X}) = Id$.\\
Conversely, Let $\mathbf{M}_{U,Y,X}$ is an invertible operator. Then
\begin{equation*}
\begin{split}
U(T_{X}\mathbf{M}^{-1}_{U,Y,X}T^{*}_{Y})
 &= U\Big(T_{X}(\mathbf{M}_{U^{-1},\tilde{X},\tilde{Y}})T^{*}_{Y}\Big)\\
 &= U\Big(T_{X}(T^{*}_{\tilde{X}}U^{-1}T_{\tilde{Y}})T^{*}_{Y}\Big)\\
 &= Id,\\
\end{split}
\end{equation*}
also $(T_{X}\mathbf{M}^{-1}_{U,Y,X}T^{*}_{Y})U = Id$. So $U$ is invertible.
\end{proof}
\begin{proposition}
Let $U$ be a bounded invertible operator on $\ell^2(A)$ and $Y = \{y_n\}\subset E$ be a modular Riesz basis. Moreover let $X = \{x_n\}$ be a
standard frame for $E$. Then the following assertions are equivalent.
\begin{enumerate}
  \item $X$ has a unique dual frame.
  \item $\mathbf{M}_{U,Y,X}$ is an invertible.
\end{enumerate}
\end{proposition}
\begin{proof}
(1)$\Rightarrow$ (2) To obtain the second statement from the first one, suppose that $X$ has a unique dual frame.
Then by \cite[Theorem 4.9]{Jing-2006}, the associated analysis operator $T_{X}$ is
surjective. Also by using the reconstruction formula \eqref{recons formula}, we conclude that $T_{X}$ is injective and so $T_{X}$ is bijective. Due to the fact
that $T^{*}_{Y}$ and $U$ are bijective, we deduce $\mathbf{M}_{U,Y,X}$ is invertible.\\
(2)$\Rightarrow$ (1) Now, to drive the first statement from the second one, we assume $\mathbf{M}_{U,Y,X}$ is invertible. Then $T_{X}$ is surjective and so by  \cite[Theorem 4.9]{Jing-2006},
$X$ has a unique dual frame.
\end{proof}

%

\bibliographystyle{amsplain}

\end{document}